\documentclass[12pt]{amsart}
\usepackage{amssymb}
\usepackage{amsmath, amscd}
\usepackage{amsthm}
\usepackage[table,xcdraw]{xcolor}
\usepackage{ulem}
\usepackage{tikz}
\usepackage{circuitikz}
\usetikzlibrary{arrows.meta, positioning, matrix}
\usepackage[colorlinks=true, linkcolor=blue,urlcolor=blue]{hyperref}

\newtheorem{theorem}{Theorem}[section]

\newtheorem{proposition}[theorem]{Proposition}
\newtheorem{lemma}[theorem]{Lemma}

\newtheorem{corollary}[theorem]{Corollary}
\theoremstyle{definition}

\newtheorem{example}[theorem]{Example}
\newtheorem{definition}[theorem]{Definition}

\newtheorem{problem}[theorem]{Problem}

\newcommand{\bigzero}{\mbox{\normalfont\Large\bfseries 0}}

\begin{document}

\author[Peter Danchev]{Peter Danchev}
\address{Institute of Mathematics and Informatics, Bulgarian Academy of Sciences, 1113 Sofia, Bulgaria}
\email{danchev@math.bas.bg; pvdanchev@yahoo.com}

\author[A. Javan]{Arash Javan}
\address{Department of Mathematics, Tarbiat Modares University, 14115-111 Tehran Jalal AleAhmad Nasr, Iran}
\email{a.darajavan@modares.ac.ir; a.darajavan@gmail.com}

\author[O. Hasanzadeh]{Omid Hasanzadeh}
\address{Department of Mathematics, Tarbiat Modares University, 14115-111 Tehran Jalal AleAhmad Nasr, Iran}
\email{o.hasanzade@modares.ac.ir; hasanzadeomiid@gmail.com}

\author[A. Moussavi]{Ahmad Moussavi}
\address{Department of Mathematics, Tarbiat Modares University, 14115-111 Tehran Jalal AleAhmad Nasr, Iran}
\email{moussavi.a@modares.ac.ir; moussavi.a@gmail.com}

\title[Rings whose non-units are weakly nil-clean]{Rings whose non-invertible elements are \\ weakly nil-clean}

\maketitle




\begin{abstract}
In regard to our recent studies of rings with (strongly, weakly) nil-clean-like properties, we explore in-depth both the structural and characterization properties of those rings whose elements that are {\it not} units are weakly nil-clean. Group rings of this sort are considered and described as well.
\end{abstract}

\section{Introduction and Motivation}

In the current paper, let \( R \) denote an associative ring with identity element, not necessarily commutative. Typically, for such a ring \( R \), the sets \( U(R) \), \( {\rm Nil}(R) \), and \( {\rm Id}(R) \) represent the set of invertible elements (i.e., the unit group of \( R \)), the set of nilpotent elements, and the set of idempotent elements in \( R \), respectively. Additionally, \( J(R) \) denotes the Jacobson radical of \( R \), and \( {\rm Z}(R) \) denotes the center of \( R \). The ring of \( n \times n \) matrices over \( R \) and the ring of \( n \times n \) upper triangular matrices over \( R \) are denoted by \( {\rm M}_n(R) \) and \( {\rm T}_n(R) \), respectively. Traditionally, a ring is termed {\it abelian} if each idempotent element is central, meaning that \( {\rm Id}(R) \subseteq {\rm Z}(R) \).

\medskip

Before we start our investigation of the characteristic properties of a newly defined by us below class of rings, we need the following background material.

\begin{definition}[\cite{2},\cite{7}]\label{definition 1.1}
Let $R$ be a ring. An element $r \in R$ is said to be {\it clean} if there is an idempotent $e \in R$ and an unit $u \in R$ such that $r=e+u$. Such an element $r$ is further called {\it strongly clean} if the existing idempotent and unit can be chosen such that $ue=eu$. A ring is called {\it clean} (respectively, {\it strongly clean}) if each of its elements is clean (respectively, strongly clean).
\end{definition}

\begin{definition}[\cite{3}]\label{definition 1.2}
An element $r$ in a ring $R$ is said to be {\it weakly clean} if there is an idempotent $e \in R$ such that $r\pm e\in U(R)$, and a weakly clean ring is defined as the ring in which every element is weakly clean.A ring $R$ is said to be {\it strongly weakly clean} provided that, for any $a \in R$, $a$ or $-a$ is strongly clean.
\end{definition}	

\begin{definition}[\cite{4}]\label{definition 1.3}
Let $R$ be a ring. An element $r \in R$ is said to be {\it nil-clean} if there is an idempotent $e \in R$ and a nilpotent $b \in R$ such that $r=e+b$. Such an element $r$ is further called {\it strongly nil-clean} if the existing idempotent and nilpotent can be chosen such that $be=eb$. A ring is called {\it nil-clean} (respectively, {\it strongly nil-clean}) if each of its elements is nil-clean (respectively, strongly nil-clean).
\end{definition}

\begin{definition}[\cite{12}, \cite{5}]\label{definition 1.4}
A ring $R$ is said to be {\it weakly nil-clean} provided that, for any $a \in R$, there exists an idempotent $e \in R$ such that $a-e$ or $a+e$ is nilpotent. A ring $R$ is said to be {\it strongly weakly nil-clean} provided that, for any $a \in R$, $a$ or $-a$ is strongly nil-clean.
\end{definition}

\begin{definition}[\cite{21}, \cite{22}]\label{definition 1.10}
A ring is called {\it UU} if all of its units are unipotent,that is, $U(R) \subseteq 1+{\rm Nil}(R)$ (and so $1+{\rm Nil}(R)=U(R)$).
\end{definition}	

\begin{definition}[\cite{13}]\label{definition 1.8}
A ring $R$ is called {\it weakly UU} and abbreviated as $WUU$ if $U(R)={\rm Nil}(R) \pm 1$. This is equivalent to the condition that every unit can be presented as either $n+1$ or $n-1$, where $n \in {\rm Nil}(R)$.
\end{definition}

\begin{definition}[\cite{14}]\label{definition 1.9}
A ring $R$ is called {\it UWNC} if every of its units is weakly nil-clean.
\end{definition}

\begin{definition}[\cite{8}]\label{definition 1.5}
A ring $R$ a {\it generalized nil-clean}, briefly abbreviated by {\it GNC}, provided $$R\backslash U(R)\subseteq {\rm Id}(R) + {\rm Nil}(R).$$	
\end{definition}

\begin{definition}[\cite{9}]\label{definition 1.6}
A ring $R$ is called {\it generalized strongly  nil-clean}, briefly abbreviated by {\it GSNC} if every non-invertible element in $R$ is strongly nil-clean.	
\end{definition}

Our aim, which motivates writing of this paper, is to examine what will happen in the dual case when non-units in rings are weakly nil-clean elements, thus somewhat also expanding weakly nil-clean rings in an other way. So, we now arrive at our key instrument introduced as follows.				

\begin{definition}\label{definition 1.7}
We call a ring $R$ a {\it generalized weakly nil-clean}, briefly abbreviated by {\it GWNC}, provided $$R\backslash U(R)\subseteq {\rm Nil}(R) \pm {\rm Id}(R).$$
\end{definition}

Now, we have the following diagram which violates the relationships between the defined above classes of rings:

\vskip1.0pc

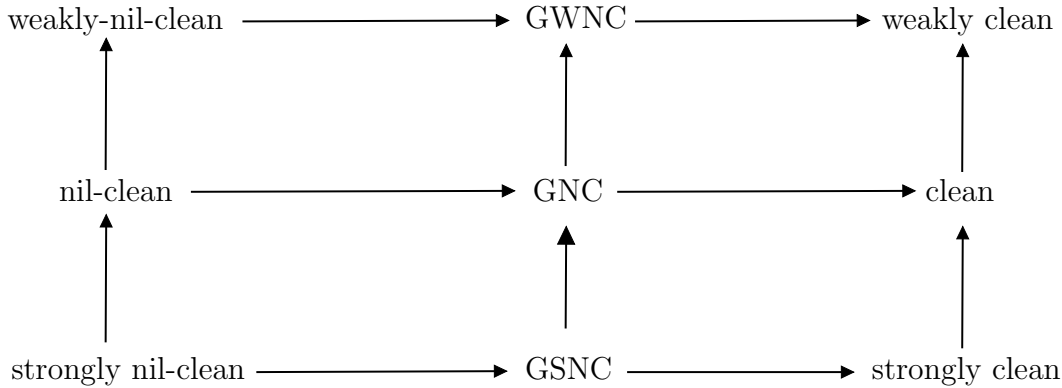
\begin{center}

\tikzset{every picture/.style={line width=0.75pt}} 

\begin{tikzpicture}[x=0.75pt,y=0.75pt,yscale=-1,xscale=1]

\draw    (179,77) -- (310.29,76.53) ;
\draw [shift={(313.29,76.52)}, rotate = 179.8] [fill={rgb, 255:red, 0; green, 0; blue, 0 }  ][line width=0.08]  [draw opacity=0] (6.25,-3) -- (0,0) -- (6.25,3) -- cycle    ;
\draw    (377,77) -- (492.29,76.53) ;
\draw [shift={(495.29,76.52)}, rotate = 179.77] [fill={rgb, 255:red, 0; green, 0; blue, 0 }  ][line width=0.08]  [draw opacity=0] (6.25,-3) -- (0,0) -- (6.25,3) -- cycle    ;
\draw    (368,163) -- (515.29,162.53) ;
\draw [shift={(518.29,162.52)}, rotate = 179.82] [fill={rgb, 255:red, 0; green, 0; blue, 0 }  ][line width=0.08]  [draw opacity=0] (6.25,-3) -- (0,0) -- (6.25,3) -- cycle    ;
\draw    (153,163) -- (313.29,162.53) ;
\draw [shift={(316.29,162.52)}, rotate = 179.83] [fill={rgb, 255:red, 0; green, 0; blue, 0 }  ][line width=0.08]  [draw opacity=0] (6.25,-3) -- (0,0) -- (6.25,3) -- cycle    ;
\draw    (186,254) -- (309.29,253.53) ;
\draw [shift={(312.29,253.52)}, rotate = 179.78] [fill={rgb, 255:red, 0; green, 0; blue, 0 }  ][line width=0.08]  [draw opacity=0] (6.25,-3) -- (0,0) -- (6.25,3) -- cycle    ;
\draw    (373,254) -- (484.29,254) ;
\draw [shift={(487.29,254)}, rotate = 180] [fill={rgb, 255:red, 0; green, 0; blue, 0 }  ][line width=0.08]  [draw opacity=0] (6.25,-3) -- (0,0) -- (6.25,3) -- cycle    ;
\draw    (110,152) -- (110.28,88.52) ;
\draw [shift={(110.29,85.52)}, rotate = 90.25] [fill={rgb, 255:red, 0; green, 0; blue, 0 }  ][line width=0.08]  [draw opacity=0] (6.25,-3) -- (0,0) -- (6.25,3) -- cycle    ;
\draw    (542,242) -- (542.28,183.52) ;
\draw [shift={(542.29,180.52)}, rotate = 90.27] [fill={rgb, 255:red, 0; green, 0; blue, 0 }  ][line width=0.08]  [draw opacity=0] (6.25,-3) -- (0,0) -- (6.25,3) -- cycle    ;
\draw    (342,232) -- (342.27,183.52) ;
\draw [shift={(342.29,180.52)}, rotate = 90.32] [fill={rgb, 255:red, 0; green, 0; blue, 0 }  ][line width=0.08]  [draw opacity=0] (8.93,-4.29) -- (0,0) -- (8.93,4.29) -- cycle    ;
\draw    (110,239) -- (110.28,177.52) ;
\draw [shift={(110.29,174.52)}, rotate = 90.26] [fill={rgb, 255:red, 0; green, 0; blue, 0 }  ][line width=0.08]  [draw opacity=0] (6.25,-3) -- (0,0) -- (6.25,3) -- cycle    ;
\draw    (542,153) -- (542.28,91.52) ;
\draw [shift={(542.29,88.52)}, rotate = 90.26] [fill={rgb, 255:red, 0; green, 0; blue, 0 }  ][line width=0.08]  [draw opacity=0] (6.25,-3) -- (0,0) -- (6.25,3) -- cycle    ;
\draw    (342,150) -- (342.28,91.52) ;
\draw [shift={(342.29,88.52)}, rotate = 90.27] [fill={rgb, 255:red, 0; green, 0; blue, 0 }  ][line width=0.08]  [draw opacity=0] (6.25,-3) -- (0,0) -- (6.25,3) -- cycle    ;

\draw (59,68) node [anchor=north west][inner sep=0.75pt]   [align=left] {weakly-nil-clean};
\draw (320,68) node [anchor=north west][inner sep=0.75pt]   [align=left] {GWNC};
\draw (500,68) node [anchor=north west][inner sep=0.75pt]   [align=left] {weakly clean};
\draw (85,155) node [anchor=north west][inner sep=0.75pt]   [align=left] {nil-clean};
\draw (324,155) node [anchor=north west][inner sep=0.75pt]   [align=left] {GNC};
\draw (522,155) node [anchor=north west][inner sep=0.75pt]   [align=left] {clean};
\draw (61,245) node [anchor=north west][inner sep=0.75pt]   [align=left] {strongly nil-clean};
\draw (320,245) node [anchor=north west][inner sep=0.75pt]   [align=left] {GSNC};
\draw (495,245) node [anchor=north west][inner sep=0.75pt]   [align=left] {strongly clean};

\end{tikzpicture}

\end{center}

Our principal work is organized as follows: In the next second section, we give some examples and suitable descriptions of certain crucial properties of GWNC rings that are mainly stated and proved in Theorems~\ref{theorem 2.35} and \ref{theorem 2.36} and the other statements associated with them. The subsequent third section is devoted to the classification when a group ring is GWNC as well as, reversely, what happens with the former objects of a group and a ring when the group ring is GWNC (see Lemma~\ref{rg} and Theorem~\ref{gr}, respectively, and the other assertions related to them). We close our work in the final fourth section with two challenging questions, namely Problems~\ref{1} and \ref{2}. 

\section{Examples and Basic Properties of GWNC Rings}

We begin with the following constructions on the definitions alluded to above.

\begin{example}\label{example 2.1}
\begin{enumerate}
\item
Any strongly nil-clean ring is GSNC, but the converse is {\it not} true in general. For instance, ${\rm M}_{2}(\mathbb{Z}_{2})$ is GSNC, but is {\it not} strongly nil-clean.
\item
Any GSNC ring is strongly clean, but the converse is {\it not} true in general. For instance,  $\mathbb{Z}_{2}[[x]]$ is strongly clean, but is {\it not} GSNC.
\item
Any nil-clean ring is GNC, but the converse is {\it not} true in general. For instance, $\mathbb{Z}_{3}$ is GNC but is {\it not} nil-clean.
\item
Any GNC ring is clean, but the converse is {\it not} true in general. For instance, $\mathbb{Z}_{6}$ is clean, but is {\it not} GNC.
\item
Any weakly nil-clean rng is GWNC, but the converse is {\it not} true in general. For instance, $\mathbb{Z}_{5}$ is GWNC, but is {\it not} weakly nil-clean.
\item
Any GWNC ring is weakly clean, but the converse is {\it not} true in general. For instance, ${\rm M}_{2}(\mathbb{Z}_{6})$ is weakly clean, but is {\it not} GWNC.
\item
Any strongly nil-clean ring is nil-clean, but the converse is {\it not} true in general. For instance, ${\rm M}_{2}(\mathbb{Z}_{2})$ is nil-clean, but is {\it not} strongly nil-clean.
\item
Any nil-clean ring is weakly nil-clean, but the converse is {\it not} true in general. For instance, $\mathbb{Z}_{3}$ is weakly nil-clean, but is {\it not} nil-clean.
\item
Any GSNC ring is GNC, but the converse is {\it not} true in general. For instance, ${\rm M}_{2}(\mathbb{Z}_{2})\oplus {\rm M}_{2}(\mathbb{Z}_{2})$ is GNC, but is {\it not} GSNC.
\item
Any GNC ring is GWNC, but the converse is {\it not} true in general. For instance, ${\rm M}_{2}(\mathbb{Z}_{3})$ is GWNC, but is {\it not} GNC.
\item
Any strongly clean ring is clean, but the converse is {\it not} true in general. For instance, ${\rm M}_{2}(\mathbb{Z}_{(2)})$ is clean, but is {\it not} strongly clean.
\item
Any clean ring is weakly clean, but the converse is {\it not} true in general. For instance, $\mathbb{Z}_{(5)}[i]$ is weakly clean, but is {\it not} clean.
\end{enumerate}
\end{example}

We continue our work with a series of technicalities.
		
\begin{lemma}\label{lemma 2.2}
Let $R$ be a ring and let $a\in R$ be a weakly nil-clean element. Then, $-a$ is weakly clean.
\end{lemma}

\begin{proof}
Assume $a = q \pm e$ is a weakly nil-clean representation. If $a=q+e$, then we have $-a=(1-e)-(q+1)$, where $1-e$ is an idempotent and $q+1$ is a unit in $R$. If $a=q-e$, then we have $-a=-(1-e)+(1-q)$, where again $1-e$ is an idempotent and $1-q$ is a unit in $R$. Thus, $-a$ has a weakly clean decomposition and hence it is a weakly clean element.
\end{proof}

\begin{corollary}\label{corollary 2.3}
Let $R$ be a GWNC ring. Then, $R$ is weakly clean.
\end{corollary}

\begin{lemma}\label{lemma 2.4}
Let $R$ be a GWNC ring. Then, $J(R)$ is nil.
\end{lemma}

\begin{proof}
Choose $j \in J(R)$. Since $j \notin U(R)$, we have $e = e^2 \in R$ and $q \in {\rm Nil}(R)$ such that $j = q \pm e$. Therefore, $$1-e = (q+1) - j \in U(R) + J(R) \subseteq U(R)$$ or $$1-e = (1-q) + j \in U(R) + J(R) \subseteq U(R),$$ so $e = 0$. Hence, $j = q \in {\rm Nil}(R)$, as required.
\end{proof}

\begin{example}\label{example 2.5}
For any ring $R$, both the polynomial ring $R[x]$ and the formal power series ring $R[[x]]$ are {\it not} GWNC rings.
\end{example}

\begin{proof}
Considering \( R[[x]] \) as a GWNC ring, we know that $$J(R[[x]]) = \{a + xf(x) : a \in J(R) \text{ and } f(x) \in R[[x]]\}.$$ So, it is evident that $x \in J(R[[x]])$. Consequently, $J(R[[x]])$ is not nil, which contradicts the assertion, thereby establishing the desired claim.

Furthermore, if $R[x]$ is GWNC, then it is weakly clean in virtue of Corollary \ref{corollary 2.3}. But this is an obvious contradiction, and hence $R[x]$ cannot be GWNC, as claimed.
\end{proof}

A ring $R$ is said to be {\it reduced} if $R$ has no non-zero nilpotent elements.

\begin{lemma}\label{lemma 2.6}
Let $R$ be a GWNC ring with $2 \in U(R)$ and, for every $u \in U(R)$, we have $u^2 = 1$. Then, $R$ is a commutative ring.
\end{lemma}

\begin{proof}
Firstly, we demonstrate that \( R \) is reduced. Assume \( R \) contains no non-trivial nilpotent elements. Suppose \( q \in \text{Nil}(R) \). Then, \( (1 \pm q) \in U(R) \), so
\[ 1 - 2q + q^2 = (1-q)^2 = 1 = (1+q)^2 = 1 + 2q + q^2. \]
Thus, \( 4q = 0 \). Since \( 2 \in U(R) \), we conclude \( q = 0 \). Hence, \( R \) is reduced and, consequently, \( R \) is abelian.

Moreover, for any \( u, v \in U(R) \), we have \( u^2 = v^2 = (uv)^2 = 1 \). Therefore, \( uv = (uv)^{-1} = v^{-1} u^{-1} = vu \), whence the units commute with each other.

In addition, let \( x, y \in R \). We consider the following cases:

1. \( x, y \in U(R) \): since the units commute, it must be that \( xy = yx \).

2. \( x, y \notin U(R) \): since \( R \) is a GWNC ring, there exist \( e, f \in \text{Id}(R) \) such that \( x = \pm e \) and \( y = \pm f \). Thus, \( xy = yx \), because \( R \) is abelian.

3. \( x \in U(R) \) and \( y \notin U(R) \): in this case, there exists \( e \in \text{Id}(R) \) such that \( y = \pm e \). Moreover, \( R \) being abelian implies \( xy = yx \).

4. \( x \notin U(R) \) and \( y \in U(R) \): similar to case (3), we can easily see that \( xy = yx \).	
\end{proof}

\begin{corollary}\label{corollary 2.7}
Let $R$ be a GWNC ring. Then, ${\rm Nil}(R)+J(R)={\rm Nil}(R)$.
\end{corollary}

\begin{proof}
Let us assume that $x \in {\rm Nil}(R)$ and $y \in J(R)$. Thus, there exists $m \in \mathbb{N}$ such that $x^m = 0$. Therefore, $(x+y)^m = x^m + j$, where $j \in J(R)$. Employing Lemma \ref{lemma 2.4}, we arrive at $(x+y)^m = j \in {\rm Nil}(R)$, as required.
\end{proof}

\begin{proposition}\label{proposition 2.8}
Let $R$ be a ring and $I$ a nil-ideal of $R$.
\begin{enumerate}
\item		
$R$ is GWNC if, and only if, $R/I$ is GWNC.
\item
A ring $R$ is GWNC if, and only if, $J(R)$ is nil and $R/J(R)$ is GWNC.
\end{enumerate}
\end{proposition}

\begin{proof}
\begin{enumerate}
\item	
We assume that $\overline{R} = R/I$ and $\bar{a} \notin U(\overline{R})$. Then, $a \notin U(R)$, so there exist $e \in {\rm Id}(R)$ and $q \in {\rm Nil}(R)$ such that $a = q \pm e$. Thus, $\bar{a} = \bar{q} \pm \bar{e}$.

Conversely, let us assume that $\overline{R}$ is a GWNC ring. We, besides, assume that $a \notin U(R)$, so $\bar{a} \notin U(\overline{R})$, and hence $\bar{a} = \bar{q} \pm \bar{e}$, where $\bar{e} \in {\rm Id}(\overline{R})$ and $\bar{q} \in {\rm Nil}(\overline{R})$. Since $I$ is a nil-ideal, we can assume that $e \in {\rm Id}(R)$ and $q \in {\rm Nil}(R)$. Therefore, $a - (q \pm e) \in I \subseteq J(R)$, so that there exists $j \in J(R)$ such that $a = (q + j) \pm e$. Hence, Corollary \ref{corollary 2.7} applies to get that $a$ has a weakly nil-clean representation, as needed.
\item
Utilizing Lemma \ref{lemma 2.4} and part (i), the conclusion is fulfilled.
\end{enumerate}
\end{proof}

\begin{corollary}\label{corollary 2.9}
Every homomorphic image of a GWNC ring is again GWNC.
\end{corollary}

\begin{proof}
It is straightforward.
\end{proof}

\begin{corollary}
Let $I$ be an ideal of a ring $R$. Then, the following are equivalent:
\begin{enumerate}
\item
$R/I$ is GWNC.
\item
$R/I^n$	is GWNC for all $n \in \mathbb{N}$.
\item
$R/I^n$ is GWNC for some $n \in \mathbb{N}$.
\end{enumerate}	
\end{corollary}

\begin{proof}
(i) $\Longrightarrow$ (ii). For any $n \in \mathbb{N}$, we know that $\dfrac{R/I^n}{I/I^n} \cong R/I$. Since $I/I^n$ is a nil-ideal of $R/I^n$ and $R/I$ is GWNC, Proposition \ref {proposition 2.8} works to derive that $R/I^n$ is a GWNC ring.\\
(ii) $\Longrightarrow$ (iii). This is quite trivial, so we leave the details.\\
(iii) $\Longrightarrow$ (i). For any ideal $I$ of $R$, we have $\dfrac{R/I^n}{I/I^n} \cong R/I$, and because we have seen above that each homomorphic image of a GWNC ring is again GWNC, we conclude that $R/I$ is GWNC.
\end{proof}

Let $Nil_{*}(R)$ denote the prime radical of a ring $R$, i.e., the intersection of all prime ideals of $R$. We know that $Nil_{*}(R)$ is a nil-ideal of $R$, and so the next assertion is immediately true.

\begin{corollary}\label{corollary 2.10}
Let $R$ be a ring. Then, the following are equivalent:
\begin{enumerate}
\item
$R$ is GWNC.
\item
$\dfrac{R}{Nil_{*}(R)}$ is GWNC.
\end{enumerate}
\end{corollary}

Let $R$ be a ring and $M$ a bi-module over $R$. The trivial extension of $R$ and $M$ is defined as
\[ {\rm T}(R, M) = \{(r, m) : r \in R \text{ and } m \in M\}, \]
with addition defined componentwise and multiplication defined by
\[ (r, m)(s, n) = (rs, rn + ms). \]
The trivial extension ${\rm T}(R, M)$ is isomorphic to the subring
\[ \left\{ \begin{pmatrix} r & m \\ 0 & r \end{pmatrix} : r \in R \text{ and } m \in M \right\} \]
of the formal $2 \times 2$ matrix ring $\begin{pmatrix} R & M \\ 0 & R \end{pmatrix}$, and also ${\rm T}(R, R) \cong R[x]/\left\langle x^2 \right\rangle$. We, moreover, note that the set of units of the trivial extension ${\rm T}(R, M)$ is
\[ U({\rm T}(R, M)) = {\rm T}(U(R), M). \]
So, as two immediate consequences, we yield:

\begin{corollary}\label{corollary 2.11}
Let $R$ be a ring and $M$ a bi-module over $R$. Then, the following hold:
\begin{enumerate}
\item
The trivial extension ${\rm T}(R, M)$ is a GWNC ring if, and only if, $R$ is a GWNC ring.
\item
For $n \geq 2$, the quotient-ring $\dfrac{R[x]}{\langle x^n\rangle}$ is a GWNC ring if, and only if, $R$ is a GWNC ring.
\item
For $n \geq 2$, the quotient-ring $\dfrac{R[[x]]}{\langle x^n\rangle}$ is a GWNC ring if, and only if, $R$ is a GWNC ring.
\end{enumerate}
\end{corollary}

\begin{proof}
\begin{enumerate}
\item
Set $A={\rm T}(R, M)$ and consider $I:={\rm T}(0, M)$. It is not too hard to verify that $I$ is a nil-ideal of $A$ such that $\dfrac{A}{I} \cong R$. So, the result follows directly from Proposition \ref{proposition 2.8}.
\item
Put $A=\dfrac{R[x]}{\langle x^n\rangle}$. Considering $I:=\dfrac{\langle x\rangle}{\langle x^n\rangle}$, we obtain that $I$ is a nil-ideal of $A$ such that $\dfrac{A}{I} \cong R$. So, the result follows automatically from Proposition \ref{proposition 2.8}.
\item
Knowing that the isomorphism $\dfrac{R[x]}{\langle x^n\rangle} \cong \dfrac{R[[x]]}{\langle x^n\rangle}$ is true, point (iii) follows at once from (ii).
\end{enumerate}
\end{proof}

\begin{corollary}\label{corollary 2.12}
Let $R$ be a ring and $M$ a bi-module over $R$. Then, the following statements are equivalent:
\begin{enumerate}
\item
$R$ is a GWNC ring.
\item
${\rm T}(R, M)$ is a GWNC ring.
\item
${\rm T}(R, R)$ is a GWNC ring.
\item
$\dfrac{R[x]}{\left\langle x^2 \right\rangle}$ is a GWNC ring.
\end{enumerate}
\end{corollary}

Consider now $R$ to be a ring and $M$ to be a bi-module over $R$. Let $${\rm DT}(R,M) := \{ (a, m, b, n) | a, b \in R, m, n \in M \}$$ with addition defined componentwise and multiplication defined by $$(a_1, m_1, b_1, n_1)(a_2, m_2, b_2, n_2) = (a_1a_2, a_1m_2 + m_1a_2, a_1b_2 + b_1a_2, a_1n_2 + m_1b_2 + b_1m_2 +n_1a_2).$$ Then, ${\rm DT}(R,M)$ is a ring which is isomorphic to ${\rm T}({\rm T}(R, M), {\rm T}(R, M))$. We also have $${\rm DT}(R, M) =
\left\{\begin{pmatrix}
	a &m &b &n\\
	0 &a &0 &b\\
	0 &0 &a &m\\
	0 &0 &0 &a
\end{pmatrix} |  a,b \in R, m,n \in M\right\}.$$ In particular, we obtain the following isomorphism as rings: $\dfrac{R[x, y]}{\langle x^2, y^2\rangle} \rightarrow {\rm DT}(R, R)$ defined by $$a + bx + cy + dxy \mapsto
\begin{pmatrix}
	a &b &c &d\\
	0 &a &0 &c\\
	0 &0 &a &b\\
	0 &0 &0 &a
\end{pmatrix}.$$
We, thereby, extract the following.

\begin{corollary}\label{corollary 2.13}
Let $R$ be a ring and $M$ a bi-module over $R$. Then, the following statements are equivalent:
\begin{enumerate}
\item
$R$ is a GWNC ring.
\item
${\rm DT}(R, M)$ is a GWNC ring.
\item
${\rm DT}(R, R)$ is a GWNC ring.
\item
$\dfrac{R[x, y]}{\langle x^2, y^2\rangle}$ is a GWNC ring.
\end{enumerate}
\end{corollary}

Let us now $\alpha$ be an endomorphism of $R$, and suppose $n$ a positive integer. It was defined by Nasr-Isfahani in \cite{10} the {\it skew triangular matrix ring} like this:
$${\rm T}_{n}(R,\alpha )=\left\{ \left. \begin{pmatrix}
	a_{0} & a_{1} & a_{2} & \cdots & a_{n-1} \\
	0 & a_{0} & a_{1} & \cdots & a_{n-2} \\
	0 & 0 & a_{0} & \cdots & a_{n-3} \\
	\ddots & \ddots & \ddots & \vdots & \ddots \\
	0 & 0 & 0 & \cdots & a_{0}
\end{pmatrix} \right| a_{i}\in R \right\}$$
with addition point-wise and multiplication, given by:
\begin{align*}
	&\begin{pmatrix}
		a_{0} & a_{1} & a_{2} & \cdots & a_{n-1} \\
		0 & a_{0} & a_{1} & \cdots & a_{n-2} \\
		0 & 0 & a_{0} & \cdots & a_{n-3} \\
		\ddots & \ddots & \ddots & \vdots & \ddots \\
		0 & 0 & 0 & \cdots & a_{0}
	\end{pmatrix}\begin{pmatrix}
		b_{0} & b_{1} & b_{2} & \cdots & b_{n-1} \\
		0 & b_{0} & b_{1} & \cdots & b_{n-2} \\
		0 & 0 & b_{0} & \cdots & b_{n-3} \\
		\ddots & \ddots & \ddots & \vdots & \ddots \\
		0 & 0 & 0 & \cdots & b_{0}
	\end{pmatrix}  =\\
	& \begin{pmatrix}
		c_{0} & c_{1} & c_{2} & \cdots & c_{n-1} \\
		0 & c_{0} & c_{1} & \cdots & c_{n-2} \\
		0 & 0 & c_{0} & \cdots & c_{n-3} \\
		\ddots & \ddots & \ddots & \vdots & \ddots \\
		0 & 0 & 0 & \cdots & c_{0}
	\end{pmatrix},
\end{align*}
where $$c_{i}=a_{0}\alpha^{0}(b_{i})+a_{1}\alpha^{1}(b_{i-1})+\cdots +a_{i}\alpha^{i}(b_{i}),~~ 1\leq i\leq n-1
.$$ We denote the elements of ${\rm T}_{n}(R, \alpha)$ by $(a_{0},a_{1},\ldots , a_{n-1})$. If $\alpha $ is the identity endomorphism, then one verifies that ${\rm T}_{n}(R,\alpha )$ is a subring of upper triangular matrix ring ${\rm T}_{n}(R)$.

\begin{corollary}\label{corollary 2.14}
Let $R$ be a ring. Then, the following are equivalent:
\begin{enumerate}
\item
$R$ is a GWNC ring.
\item
${\rm T}_{n}(R,\alpha )$ is a GWNC ring.
\end{enumerate}
\end{corollary}

\begin{proof}
Choose
	$$I:=\left\{
	\left.
	\begin{pmatrix}
		0 & a_{12} & \ldots & a_{1n} \\
		0 & 0 & \ldots & a_{2n} \\
		\vdots & \vdots & \ddots & \vdots \\
		0 & 0 & \ldots & 0
	\end{pmatrix} \right| a_{ij}\in R \quad (i\leq j )
	\right\}.$$
Then, one easily inspects that $I^{n}=\{0\}$ and that $\dfrac{{\rm T}_{n}(R,\alpha )}{I} \cong R$. Consequently, we apply Proposition \ref{proposition 2.8} to receive the desired result.
\end{proof}

Let us now again $\alpha$ be an endomorphism of $R$. We denote by $R[x,\alpha ]$ the {\it skew polynomial ring} whose elements are the polynomials over $R$, the addition is defined as usual, and the multiplication is defined by the equality $xr=\alpha (r)x$ for any $r\in R$. So, there is a ring isomorphism $$\varphi : \dfrac{R[x,\alpha]}{\langle x^{n}\rangle }\rightarrow {\rm T}_{n}(R,\alpha),$$ given by $$\varphi (a_{0}+a_{1}x+\ldots +a_{n-1}x^{n-1}+\langle x^{n} \rangle )=(a_{0},a_{1},\ldots ,a_{n-1})$$ with $a_{i}\in R$, $0\leq i\leq n-1$. Thus, one finds that ${\rm T}_{n}(R,\alpha )\cong \dfrac{R[x,\alpha ]}{\langle  x^{n}\rangle}$, where $\langle x^{n}\rangle$ is the ideal generated by $x^{n}$.

\medskip

We, thereby, detect the following claim.

\begin{corollary}\label{corollary 2.15}
Let $R$ be a ring with an endomorphism $\alpha$ such that $\alpha (1)=1$. Then, the following are equivalent:
\begin{enumerate}
\item
$R$ is a GWNC ring.
\item
$\dfrac{R[x,\alpha ]}{\langle x^{n}\rangle }$ is a GWNC ring.
\item
$\dfrac{R[[x,\alpha ]]}{\langle x^{n}\rangle }$ is a GWNC ring.
\end{enumerate}
\end{corollary}

Assuming now that \({\rm L}_n(R) = \left\{
\begin{pmatrix}
    0 & \cdots & 0 & a_1 \\ 0 & \cdots & 0 & a_2 \\ \vdots & \ddots & \vdots & \vdots \\ 0 & \cdots & 0 & a_n
\end{pmatrix}
\in {\rm T}_n(R) : a_i \in R \right \} \subseteq {\rm T}_n(R)\) and \({\rm S}_n(R) = \{(a_{ij}) \in {\rm T}_n(R) : a_{11}= \cdots = a_{nn}\} \subseteq {\rm T}_n(R)\), it is not so difficult to check that the mapping $\varphi: {\rm S}_n(R) \to {\rm T}({\rm S}_{n-1}(R), {\rm L}_{n-1}(R))$, defined as
$$\varphi \left(
\begin{pmatrix}
    a_{11} & a_{12} & \cdots & a_{1n}\\
    0 & a_{11} & \cdots & a_{2n}\\
    \vdots  & \vdots & \ddots  & \vdots \\
    0 & 0 & \cdots & a_{11}
\end{pmatrix}
\right) =
\begin{pmatrix}
    a_{11} & a_{12} & \cdots & a_{1,n-1} & 0 & \cdots & 0 & a_{1n} \\
    0 & a_{11} & \cdots & a_{2,n-1} & 0 & \cdots & 0 & a_{2n} \\
    \vdots  & \vdots  & \ddots  & \vdots  & \vdots  & \ddots  & \vdots  & \vdots  \\
    0 & 0 & \cdots & a_{11} & 0 & \cdots & 0 & a_{n-1,n} \\
    0 & 0 & \cdots & 0 & a_{11} & a_{12} & \cdots & a_{1,n-1} \\
    0 & 0 & \cdots & 0 & 0 & a_{11} & \cdots & a_{2,n-1} \\
    \vdots & \vdots & \ddots  & \vdots & \vdots  & \vdots  & \ddots  & \vdots \\
    0 & 0 & \cdots & 0 & 0 & 0 & \cdots & a_{11} \\
\end{pmatrix},$$
gives ${\rm S}_n(R) \cong {\rm T}({\rm S}_{n-1}(R), {\rm L}_{n-1}(R))$. Notice that this isomorphism is a helpful instrument to study the ring ${\rm S}_n(R)$, because by examining the trivial extension and using induction on $n$, we can extend the result to ${\rm S}_n(R)$.

\medskip

Specifically, we are able to establish truthfulness of the following two statements.

\begin{corollary}\label{corollary 2.16}
Let $R$ be ring. Then, the following items hold:
\begin{enumerate}
\item
For $n \ge 2$, ${\rm S}_n(R)$ is GWNC ring if, and only if, $R$ is a GWNC.
\item
For $n,m \ge 2$, ${\rm A}_{n,m}(R):=R[x,y \mid x^n=yx=y^m=0]$ is GWNC ring if, and only if, $R$ is a GWNC.
\item
For $n,m \ge 2$, ${\rm B}_{n,m}(R):=R[x,y \mid x^n=y^m=0]$ is GWNC ring if, and only if, $R$ is a GWNC.
\end{enumerate}
\end{corollary}

\begin{proof}
\begin{enumerate}
\item	
We assume $I:=\{(a_{ij}) \in {\rm S}_n(R) : a_{11}=0\}$, so evidently $I$ is a nil-ideal of ${\rm S}_n(R)$, and therefore we derive ${\rm S}_n(R)/I \cong R$.
\item
We set $$I:=\{a+\sum_{i=1}^{n-1} b_{i}x^i+ \sum_{j=1}^{m-1}c_{j}y^j \in {\rm A}_{n,m}(R) : a=0\},$$ so apparently $I$ is a nil-ideal of ${\rm A}_{n,m}(R)$, and thus we infer ${\rm A}_{n,m}(R)/I \cong R$.
\item
We put $$I:=\{\sum_{i=0}^{n-1}\sum_{j=0}^{m-1} a_{ij}x^iy^j \in {\rm B}_{n,m}(R) : a_{00}=0\},$$ so elementarily $I$ is a nil-ideal of ${\rm B}_{n,m}(R)$, and so we deduce ${\rm B}_{n,m}(R)/I \cong R$.
\end{enumerate}
\end{proof}

In the other vein, Wang introduced in \cite{11} the matrix ring ${\rm S}_{n,m}(R)$ for a given ring $R$. Then, the matrix ring $S_{n,m}(R)$ can be represented as
$$\left\{ \begin{pmatrix}
	a & b_1 & \cdots & b_{n-1} & c_{1n} & \cdots & c_{1 n+m-1}\\
	\vdots  & \ddots & \ddots & \vdots & \vdots & \ddots & \vdots \\
	0 & \cdots & a & b_1 & c_{n-1,n} & \cdots & c_{n-1,n+m-1} \\
	0 & \cdots & 0 & a & d_1 & \cdots & d_{m-1} \\
	\vdots  & \ddots & \ddots & \vdots & \vdots & \ddots & \vdots \\
	0 & \cdots & 0 & 0  & \cdots & a & d_1 \\
	0 & \cdots & 0 & 0  & \cdots & 0 & a
\end{pmatrix}\in {\rm T}_{n+m-1}(R) : a, b_i, d_j,c_{i,j} \in R \right\}.$$
Also, let ${\rm T}_{n,m}(R)$ be
$$\left\{ \left(\begin{array}{@{}c|c@{}}
	\begin{matrix}
		a & b_1 & b_2 & \cdots & b_{n-1} \\
		0 & a & b_1 & \cdots & b_{n-2} \\
		0 & 0 & a & \cdots & b_{n-3} \\
		\vdots & \vdots & \vdots & \ddots & \vdots \\
		0 & 0 & 0 & \cdots & a
	\end{matrix}
	& \bigzero \\
	\hline
	\bigzero &
	\begin{matrix}
		a & c_1 & c_2 & \cdots & c_{m-1} \\
		0 & a & c_1 & \cdots & c_{m-2} \\
		0 & 0 & a & \cdots & c_{m-3} \\
		\vdots & \vdots & \vdots & \ddots & \vdots \\
		0 & 0 & 0 & \cdots & a
	\end{matrix}
\end{array}\right)\in {\rm T}_{n+m}(R) : a, b_i,c_j \in R \right\}, $$
and let
$${\rm U}_{n}(R)=\left\{ \begin{pmatrix}
	a & b_1 & b_2 & b_3 & b_4 & \cdots & b_{n-1} \\
	0 & a & c_1 & c_2 & c_3 & \cdots & c_{n-2} \\
	0 & 0 & a & b_1 & b_2 & \cdots & b_{n-3} \\
	0 & 0 & 0 & a & c_1 & \cdots & c_{n-4} \\
	\vdots & \vdots & \vdots & \vdots &  &  & \vdots \\
	0 &0 & 0 & 0 & 0 & \cdots & a
\end{pmatrix}\in {\rm T}_{n}(R) :  a, b_i, c_j \in R \right\}.$$	

Thus, we come to the following assertion.

\begin{corollary}\label{corollary 2.17}
Let $R$ be a ring. Then, the following statements are equivalent:
\begin{enumerate}
\item
$R$ is a GWNC ring.
\item
${\rm S}_{n,m}(R)$ is a GWNC ring.
\item
${\rm T}_{n,m}(R)$ is a GWNC ring.
\item
${\rm U}_{n}(R)$ is a GWNC ring.
\end{enumerate}	
\end{corollary}

Let us now we recollect that a ring $R$ is called {\it local}, provided $R/J(R)$ is a division ring, that is, every element of $R$ lies in either $U(R)$ or $J(R)$.

\medskip

We are now in establishing a series of preliminary claims before formulating the major assertion.
	
\begin{proposition}\label{proposition 2.18}
Let $R$ be a ring with only trivial idempotents. Then, $R$ is GWNC if, and only if, $R$ is a local ring with $J(R)$ nil.
\end{proposition}

\begin{proof}
Assuming $R$ is a GWNC ring, Lemma \ref{lemma 2.4} insures that $J(R)$ is nil. Now, if $a \notin U(R)$, then we have either $a = q \pm 1$ or $a = q \pm 0$, where $q \in {\rm Nil}(R)$. Since $a$ is not a unit, it must be that $a = q \pm 0$, implying $a = q \in {\rm Nil}(R)$. So, according to \cite[Proposition 19.3]{1}, $R$ is a local ring.

Conversely, suppose $R$ is a local ring with nil Jacobson radical $J(R)$. So, for each $a \notin U(R)$, we have $a \in J(R) \subseteq {\rm Nil}(R)$, whence $a$ is a weakly nil-clean element, as required.
\end{proof}

\begin{proposition}\label{proposition 2.19}
Let $R$ and $S$ be rings. If $R\times S$ is GWNC, then $R$ and $S$ are weakly nil-clean.
\end{proposition}

\begin{proof}
Let $a\in R$ is an arbitrary element, so $(a,0) \in R \times S$ is not unit. Then, we have $(a,0)=(q,0)\pm(e,0)$, where $(q,0)$ is a nilpotent and $(e,0)$ is an idempotent in $R \times S$, whence $a=q\pm e$, where $q$ is a nilpotent and $e$ is an idempotent in $R$. So, $a$ is a weakly nil-clean element. Thus, $R$ is weakly nil-clean ring. Similarly, $S$ is a weakly nil-clean ring.
\end{proof}

\begin{proposition}\label{proposition 2.20}	
Let $R_{i}$ be a ring for all $i\in I$. If $\prod^{n}_{i=1}R_{i}$ is GWNC, then each $R_{i}$ is GWNC.	
\end{proposition}

\begin{proof}
It is immediate referring to Corollary \ref{corollary 2.9}.	
\end{proof}

\begin{proposition}\label{proposition 2.21}
The direct product $\prod_{i=1}^{n} R_i$ is GWNC for $n \ge 3$ if, and only if, each $R_i$ is weakly nil-clean and at most one of them is {\it not} nil-clean.
\end{proposition}

\begin{proof}
($\Rightarrow$). Assume \(\prod_{i=1}^{n} R_i\) is a GWNC ring. Therefore, by Proposition \ref{proposition 2.19}, \(\prod_{i=1}^{n-1} R_i\) and \(R_n\) are weakly nil-clean rings. Thus, owing to \cite[Proposition 3]{12}, with no loss of generality, we may assume that, for each \(1 \le i \le n-2\), \(R_i\) is a nil-clean ring. Again, since \[\prod_{i=1}^{n} R_i = (R_1 \times \cdots \times R_{n-2}) \times (R_{n-1} \times R_n)\], Proposition \ref{proposition 2.19} implies that \(R_{n-1} \times R_n\) is weakly nil-clean. Therefore, \cite[Proposition 3]{12} allows us to assume that \(R_{n-1}\) is nil-clean and \(R_n\) is weakly nil-clean, as needed.\\
($\Leftarrow$). It follows directly from \cite[Proposition 3]{12}.
\end{proof}

\begin{example}\label{example 2.22}
The ring $\mathbb{Z}_3 \times \mathbb{Z}_3$ is GWNC, but $\mathbb{Z}_3$ is {\it not} nil-clean. The ring $\mathbb{Z}_6$ is weakly nil-clean, but $\mathbb{Z}_6 \times \mathbb{Z}_6$ is {\it not} GWNC. 	
\end{example}

\begin{corollary}\label{corollary 2.23}
Let $L=\prod_{i \in I} R_i$ be the direct product of rings $R_i \cong R$ and $|I| \geq 3$. Then, $L$ is a GWNC ring if, and only if, $L$ is a GNC ring if, and only if, $L$ is nil-clean if, and only if, $R$ is nil-clean.
\end{corollary}

\begin{corollary}\label{corollary 2.24}
For any $n \geq 3$, the ring $R^n$ is GWNC if, and only if, $R^n$ is GNC if, and only if, $R$ is nil-clean.
\end{corollary}

\begin{proposition}\label{proposition 2.25}
Let $R$ be a ring. Then, the following are equivalent:
\begin{enumerate}
\item
$R$ is nil-clean.
\item
${\rm T}_{n}(R)$ is weakly nil-clean for all $n \in \mathbb{N}$.
\item
${\rm T}_n(R)$ is weakly nil-clean for some $n \geq 3$.
\item
${\rm T}_n(R)$ is GWNC for some $n \geq 3$.
\end{enumerate}
\end{proposition}

\begin{proof}
(i) $\Rightarrow$ (ii). This follows employing \cite[Theorem 4.1]{4}.\\
(ii) $\Rightarrow$ (iii) $\Rightarrow$ (iv). These two implications are trivial, so we remove the details.\\
(iv) $\Rightarrow$ (i). Setting $I:=\{ (a_{ij})\in {\rm T}_{n}(R)| a_{ii}=0\}$, we obtain that it is a nil-ideal in ${\rm T}_{n}(R)$ with $\dfrac{{\rm T}_{n}(R)}{I}\cong R^{n}$. Therefore, Corollary \ref{corollary 2.24} is applicable to get the pursued result.
\end{proof}

\begin{example}\label{example 2.26}
The ring ${\rm T}_{2}(\mathbb{Z}_3)$ is GWNC, but $\mathbb{Z}_3$ is {\it not} nil-clean. The ring $\mathbb{Z}_6$ is weakly nil-clean, but ${\rm T}_{2}(\mathbb{Z}_6)$ is {\it not} GWNC.  	
\end{example}

\begin{lemma}\label{lemma 2.27}
Let $R$ be a ring and $2 \in J(R)$. Then, the following two points are equivalent:
\begin{enumerate}
\item
$R$ is a GWNC ring.
\item
$R$ is a GNC ring.
\end{enumerate}
\end{lemma}

\begin{proof}
(ii) $\Longrightarrow$ (i). It is straightforward.\\
(i) $\Longrightarrow$ (ii). Note that $\dfrac{R}{J(R)}$ is of characteristic $2$, because $2 \in J(R)$, and so $a=-a$ for each $a \in \dfrac{R}{J(R)}$. That is why, $\dfrac{R}{J(R)}$ is a GNC ring, and so we can invoke \cite[Proposition 2.11]{8} as $J(R)$ is nil in virtue of Proposition \ref{proposition 2.8}.
\end{proof}

A ring $R$ is said to be {\it exchange} if, for any $a\in R$, there exists an idempotent $e\in aR$ such that $1-e\in (1-a)R$ (see, e.g., \cite{2}). Notice that every clean ring is exchange, whereas the converse is true in the abelian case (see \cite[Proposition 1.8]{2}). Moreover, a ring $R$ is said to be {\it weakly exchange} if, for any $a\in R$, there exists an idempotent $e\in aR$ such that $1-e\in (1-a)R$ or $1-e\in (1+a)R$ (see, e.g., \cite{20}). Note that any weakly clean ring is weakly exchange, while the converse is valid for abelian rings (see \cite[Theorem 2.1]{20}).
	
\begin{lemma}\label{lemma 2.28}
Let $R$ be a ring and. Then, the following are equivalent:
\begin{enumerate}
\item
$R$ is a strongly weakly nil-clean ring.
\item
$R$ is both WUU and GWNC.
\end{enumerate}
\end{lemma}

\begin{proof}
(i) $\Longrightarrow$ (ii). We know that each strongly weakly nil-clean ring is weakly nil-clean, and hence is GWNC. Also, every strongly weakly nil-clean ring is WUU appealing to \cite[Proposition 2.1]{14}.\\	
(ii) $\Longrightarrow$ (i). We know that each GWNC ring is weakly clean, whence is weakly exchange, so the conclusion follows from \cite[Theorem 3.6]{19}.
\end{proof}	

\begin{lemma}\label{lemma 2.29}
A ring $R$ is strongly nil-clean if, and only if,
\begin{enumerate}
\item
$R$ is GWNC

\noindent and

\item
$R$ is an UU ring.
\end{enumerate}
\end{lemma}

\begin{proof}
It is immediate from the combination of \cite[Corollary 3.4]{19} and Lemma \ref{lemma 2.4}.
\end{proof}	

A ring $R$ is said to be {\it strongly $\pi$-regular} provided that, for any $a\in R$, there exist natural number $n$ such that $a^n\in a^{n+1}R$. A ring $R$ is called {\it semi-potent} if every one-sided ideal not contained in $J(R)$ contains a non-zero idempotent.

\medskip

We now have the following coincidences. 

\begin{corollary}\label{corollary 2.30}
Let $R$ be an UU ring. Then, the following are equivalent:
\begin{enumerate}
\item
$R$ is a strongly clean ring.
\item
$R$ is a strongly nil-clean ring.
\item
$R$ is a GSNC ring.
\item
$R$ is a strongly $\pi$-regular ring.
\item
$R$ is a GNC ring.
\item
$R$ is a GWNC ring.
\item
$R$ is a semi-potent ring.
\item
$R$ is a weakly clean ring.
\item
$R$ is a weakly exchange ring.
\end{enumerate}
\end{corollary}

\begin{proof}
(i), (ii), (iii), and (iv) are equivalent by \cite[Corollary 2.20]{9}.\\
(ii) and (v) are equivalent via \cite[Corollary 2.28]{8}.\\
(ii) and (vi) are equivalent by Lemma \ref{lemma 2.29}.\\
(ii) and (vii) are equivalent via \cite[Theorem 2.25]{23}.\\
(vi) $\Longrightarrow$ (viii). It is elementary thanks to Corollary \ref{corollary 2.3}.\\
(viii) $\Longrightarrow$ (vi). Let $R$ is a weakly clean ring and let $a\in R$, so we have $a+1=u\pm e$, where $u$ is a unit in $R$ and $e$ is an idempotent in $R$. Thus, $a=(u-1)\pm e$, where $u-1$ is a nilpotent element in $R$. So, $R$ is a weakly nil-clean ring and hence is GWNC ring.\\
(ix) and (ii) are equivalent under validity of \cite[Theorem 2.4]{22} and \cite[Corollary 3.4]{19}.
\end{proof}

\begin{lemma}\label{lemma 2.31}
Let $R$ be a local ring. Then, the following are equivalent:
\begin{enumerate}
\item
$R$ is a GWNC ring.
\item
$R$ is a GNC ring.
\end{enumerate}
\end{lemma}	

\begin{proof}
It is routine that every division ring is GNC, so the result is concluded exploiting Lemma \ref{lemma 2.4}.
\end{proof}	
	
\begin{lemma}\cite[Lemma 24]{12} \label{lemma 2.32}
Let $D$ be a division ring. If $|D| \ge 4$ and $a \in D \backslash \{0, 1, -1\}$, then
$\begin{pmatrix}
    a & 0 \\0 & 0
\end{pmatrix} \in M_n(D)$
is not weakly nil-clean.
\end{lemma}

\begin{lemma} \label{lemma 2.33}
Let $n\ge 2$ and let $D$ be a division ring. Then, the matrix ring ${\rm M}_n(D)$ is a GWNC ring if, and only if, either $D \cong \mathbb{Z}_2$ or $D \cong \mathbb{Z}_3$, $n=2$.
\end{lemma}

\begin{proof}
If $D \cong \mathbb{Z}_2$, then ${\rm M}_n(D)$ is nil-clean and hence is GWNC. If, however, $D \cong \mathbb{Z}_3$ and $n=2$, then ${\rm M}_2(\mathbb{Z}_3)$ is GWNC ring.

Oppositely, if ${\rm M}_n(D)$ is a GWNC ring and $|D| \ge 4$, then Lemma \ref{lemma 2.32} gives that, for every $a \in D \backslash \{0,1,-1\}$, the element $\begin{pmatrix}
    a & 0 \\0 & 0
\end{pmatrix} \in {\rm M}_n(D)$ is a non-unit in ${\rm M}_n(D)$ which is not weakly nil-clean, leading to a contradiction. Therefore, it must be that $|D|= 2$ or $|D|=3$. If, foremost, $|D|= 2$, then $D \cong \mathbb{Z}_2$. If, next, $|D|= 3$, then $D \cong \mathbb{Z}_3$. Besides, we have $n=2$, as for otherwise, let $A_{11}:=\begin{pmatrix}
1 & 0 \\0 & -1
\end{pmatrix} \in {\rm M}_2(\mathbb{Z}_3)$, so we have $A:=\begin{pmatrix}
A_{11} & 0 \\0 & 0
\end{pmatrix} \in {\rm M}_n(\mathbb{Z}_3)$ is not weakly nil-clean for all $n\ge 3$ arguing as in the proof of \cite[Theorem 25]{12}. Likewise, one inspects that $A$ is not a unit, a contradiction.
\end{proof}

As two consequences, we extract:

\begin{corollary}\label{corollary 2.34}
Let $n\ge 3$ and let $D$ be a division ring. Then, the matrix ring ${\rm M}_n(D)$ is a GWNC ring if, and only if, $D \cong \mathbb{Z}_2$.	
\end{corollary}

\begin{corollary} \label{cor nontriv idem}
Let $R$ be a ring with no non-trivial idempotents and \( n \ge 2 \). Then, the following conditions are equivalent:

(1) \( {\rm M}_n(R) \) is a GWNC ring.

(2) Either \( R/J(R) \cong \mathbb{Z}_2 \) for \( n \ge 2 \) and \( {\rm M}_n(J(R)) \) is nil, or \( R/J(R) \cong \mathbb{Z}_3 \) for \( n = 2 \) and \( {\rm M}_n(J(R)) \) is nil.
\end{corollary}

\begin{proof}
Assume \( {\rm M}_n(R) \) is a GWNC ring. We show that \( R \) is local. Let \( a \in R \). Consider \( A := a e_{11} \notin {\rm GL}_n(R) \). Thus, there exist \( E\in \text{Id}({\rm M}_n(R)) \) and \( Q \in \text{Nil}({\rm M}_n(R)) \) such that \( A = E + Q \) or \( A = -E + Q \). Assuming first that \( A = E + Q \), then \( -A = (I_n - E) - (Q + I_n) \). Let \( U := I_n + Q \in {\rm GL}_n(R) \). Therefore, \[ -U^{-1}A = U^{-1}(I_n - E)UU^{-1} - I_n \]. Let \( F = U^{-1}(I_n - E)U \). Hence, \( -(I_n - F)U^{-1}A = -(I_n - F) \). So,
\[ I_n - F = \begin{pmatrix}
    e \ 0 \ \cdots \ 0 \\ \ast \ 0 \cdots \ 0 \\
    \vdots \ \vdots \ \ddots \ \vdots \\ \ast \ 0 \ \cdots \ 0
\end{pmatrix}
, \]
where \( e \in \{0, 1\} \) since \( R \) is a ring with no non-trivial idempotents. If, firstly, \( e = 0 \), then \( I_n - F = 0 \) implying \( F = I_n \). Since \( F = U^{-1}(I_n - E)U \), we get \( E = 0 \), hence \( A = Q \in \text{Nil}(M_n(R)) \), so \( a \in \text{Nil}(R) \) forcing \( 1 - a \in U(R) \).

Assume next that \( e = 1 \), then
\( F = \begin{pmatrix}
    0 \ 0 \ \cdots \ 0 \\ \ast \ 1 \cdots \ 0 \\
    \vdots \ \vdots \ \ddots \ \vdots \\ \ast \ 0 \ \cdots \ 1
\end{pmatrix}
 \).
Choosing \( U^{-1} = (v_{ij}) \), and bearing in mind \( -U^{-1}A = FU^{-1} - I_n \), we have \( v_{11}a = 1 \). Moreover, since \( R \) is a ring with no non-trivial idempotents, and \( av_{11} \in \text{Id}(R) \), either \( av_{11} = 0 \) or \( av_{11} = 1 \). If \( av_{11} = 0 \), from \( v_{11}a = 1 \), we deduce \( a = 0 \), a contradiction. Therefore, \( av_{11} = 1 \).

Now, suppose \( A = -E + Q \). Then, \( A = (I - E) + (Q - I) \). Assuming \( U = Q - I \), and similarly to the above arguments, we can demonstrate that either \( a \in U(R) \) or \( 1 - a \in U(R) \) yielding that \( R \) is a local ring.

Since \( {\rm M}_n(R) \) is a GWNC ring, it follows that \( {\rm M}_n(R/J(R)) \) is also GWNC. Taking into account Lemma \ref{lemma 2.33}, for \( n \ge 3 \) we get \( R/J(R) \cong \mathbb{Z}_2 \), and for \( n = 2 \) we get \( R/J(R) \cong \mathbb{Z}_3 \). Additionally, with the help of Proposition \ref{proposition 2.8}, \(J({\rm M}_n(R))= {\rm M}_n(J(R)) \) is nil.
\end{proof}

Recall that a ring is {\it Boolean} if every its element is an idempotent.

\medskip

We now have all the ingredients necessary to establish the following two main results.

\begin{theorem}\label{theorem 2.35}
Let $R$ be a commutative ring. Then, ${\rm M}_n(R)$ is GWNC if, and only if, ${\rm M}_n(R)$ is nil-clean for all $n\ge3$.
\end{theorem}

\begin{proof}
($\Longrightarrow$). Let $M$ be a maximal ideal of $R$ and $n\ge3$. Hence, ${\rm M}_n(R/M)$ is GWNC. Since $R/M$ is a field, it follows from Corollary \ref{corollary 2.34} that $R/M \cong \mathbb{Z}_2$. Thus, $R/J(R)$ is isomorphic to the subdirect product of $\mathbb{Z}_2$'s; whence, $R/J(R)$ is Boolean. Employing \cite[Corollary 6]{24}, ${\rm M}_n(R/J(R))$ is nil-clean. Apparently, $J({\rm M}_n(R))$ is nil. Accordingly, ${\rm M}_n(R)$ is nil-clean in view of \cite[Corollary 3.17]{4}.\\
($\Longleftarrow$). This is obvious, so we omit the details.
\end{proof}

Recall that a ring $R$ is said to be {\it semi-local} if $R/J(R)$ is a left artinian ring or, equivalently, if $R/J(R)$ is a semi-simple ring.

\begin{theorem}\label{theorem 2.36}
Let $R$ be a ring. Then, the following conditions are equivalent for a semi-local ring:
\begin{enumerate}
\item
$R$ is a GWNC ring.
\item
Either $R$ is a local ring with a nil Jacobson radical, or $R/J(R) \cong {\rm M}_2(\mathbb{Z}_3)$ with a nil Jacobson radical, or $R/J(R) \cong \mathbb{Z}_3 \times \mathbb{Z}_3$ with a nil Jacobson radical, or $R$ is a weakly nil-clean ring.
\end{enumerate}
\end{theorem}

\begin{proof}
(ii) $\Longrightarrow$ (i). The proof is straightforward by combining Lemma \ref{lemma 2.33} and Proposition \ref{proposition 2.8}. Also, we know that $\mathbb{Z}_3 \times \mathbb{Z}_3$ is a GWNC ring.\\
(i) $\Longrightarrow$ (ii). Since $R$ is semi-local, $R/J(R)$ is semi-simple, so we have $$R/J(R) \cong \prod_{i=1}^{m} {\rm M}_{n_i}(D_i),$$ where each $D_i$ is a division ring. Moreover, the application of Proposition \ref{proposition 2.8} leads to $J(R)$ is nil, and $R/J(R)$ is a GWNC ring. If $m = 1$, then Lemma \ref{lemma 2.33} applies to get that either $R/J(R) = D_1$ or $R/J(R) \cong {\rm M}_{n_1}(\mathbb{Z}_2)$ or $R/J(R) \cong {\rm M}_2(\mathbb{Z}_3)$.

However, we know that ${\rm M}_{n_1}(\mathbb{Z}_2)$ is nil-clean and hence is weakly nil-clean, so $R/J(R)$ is weakly nil-clean. As $J(R)$ is nil, $R$ is weakly nil-clean. If $m = 2$, so $$R/J(R) \cong {\rm M}_{n_1}(D_1) \times {\rm M}_{n_2}(D_2).$$ As $R/J(R)$ is GWNC, both ${\rm M}_{n_1}(D_1)$ and ${\rm M}_{n_2}(D_2)$ are weakly nil-clean using Proposition \ref {proposition 2.19}. Thus, $D_1 \cong \mathbb{Z}_2$, or $D_1 \cong \mathbb{Z}_3$ and $n_1=1$; $D_2 \cong \mathbb{Z}_2$, or $D_2 \cong \mathbb{Z}_3$ and $n_2=1$. Consequently, we have $$R/J(R) \cong {\rm M}_{n_1}(\mathbb{Z}_2) \times {\rm M}_{n_2}(\mathbb{Z}_2),$$ or $$R/J(R) \cong {\rm M}_{n_1}(\mathbb{Z}_2) \times \mathbb{Z}_3,$$ or $$R/J(R) \cong \mathbb{Z}_3 \times \mathbb{Z}_3.$$ Knowing that ${\rm M}_{n_1}(\mathbb{Z}_2) \times {\rm M}_{n_2}(\mathbb{Z}_2)$ is nil-clean, so $R/J(R)$ is nil-clean, and hence is weakly nil-clean. As $J(R)$ is nil, $R$ is weakly nil-clean (see \cite{12}).

But, we also know enabling from \cite{12} that ${\rm M}_{n_1}(\mathbb{Z}_2) \times \mathbb{Z}_3$ is weakly nil-clean and hence $R/J(R)$ is too weakly nil-clean. As $J(R)$ is nil, as above, $R$ is weakly nil-clean. If $m > 2$, then Proposition \ref {proposition 2.21} employs to derive that each ${\rm M}_{n_i}(D_i)$ is weakly nil-clean and at most one of them is not nil-clean. Finally, referring to \cite[Theorem 25]{12} and \cite[Theorem 3]{25}, for any $1 \le i \le m$, we deduce $D_i \cong \mathbb{Z}_2$ and there exist an index, say $j$, with $D_j \cong \mathbb{Z}_2$, or $D_j \cong \mathbb{Z}_3$ and $n=1$. Therefore, \cite[Corollary 26]{12} applies to conclude that $R$ is a weakly nil-clean ring, as expected.
\end{proof}

Two more consequences sound like these.

\begin{corollary}\label{corollary 2.37}
Let $R$ be a ring. Then, the following conditions are equivalent for a semi-simple ring:
\begin{enumerate}
\item
$R$ is a GWNC ring.
\item
Either $R$ is a division ring, or $R\cong {\rm M}_2(\mathbb{Z}_3)$, or $R\cong \mathbb{Z}_3 \times \mathbb{Z}_3$, or $R$ is a weakly nil-clean ring.
\end{enumerate}
\end{corollary}

\begin{corollary}\label{corollary 2.38}
Let $R$ be a ring. Then, the following conditions are equivalent for an artinian (in particular, a finite) ring:
\begin{enumerate}
\item
$R$ is a GWNC ring.
\item
Either $R$ is a local ring with a nil Jacobson radical, or $R/J(R) \cong {\rm M}_2(\mathbb{Z}_3)$ with a nil Jacobson radical, or $R/J(R) \cong \mathbb{Z}_3 \times \mathbb{Z}_3$  with a nil Jacobson radical, or $R$ is a weakly nil-clean ring.
\end{enumerate}
\end{corollary}

It is long known that a ring $R$ is called {\it $2$-primal} if its lower nil-radical $Nil_{*}(R)$ consists precisely of all the nilpotent elements of $R$. For instance, it is well known that both reduced rings and commutative rings are both $2$-primal.

\begin{proposition}\label{proposition 2.41}
Let $R$ be a $2$-primal ring and $n \ge 3$. Then, ${\rm M}_n(R)$ is GWNC if, and only if, $R/J(R)$ is Boolean and $J(R)$ is nil.
\end{proposition}

\begin{proof}
$(\Longleftarrow)$. Invoking \cite[Theorem 6.1]{26}, we conclude that ${\rm M}_n(R)$ is nil-clean, whence is weakly nil-clean, so that it is GWNC.\\
$(\Longrightarrow)$. Since ${\rm M}_n(R)$ is GWNC, one follows that ${\rm M}_n(R)/J({\rm M}_n(R)) \cong {\rm M}_n(R/J(R))$ is GWNC appealing to Corollary \ref {corollary 2.9}, and $J({\rm M}_n(R))={\rm M}_n(J(R))$ is nil appealing to Proposition \ref {proposition 2.8}. It now follows that $J(R)$ is nil. But since $R$ is $2$-primal, it also follows that ${\rm Nil}_{*}(R) = J(R) = {\rm Nil}(R)$ and hence $R/J(R)$ is a reduced ring. Therefore, $R/J(R)$  is a sub-direct product of a family of domains $\{S_i\}_{i \in I}$. As being an image of ${\rm M}_n(R/J(R))$, the matrix ring ${\rm M}_n(S_i)$ is also GWNC. Therefore, Corollary \ref{cor nontriv idem} allows us to obtain that, for each \( i \in I \), \( S_i / J(S_i) \cong \mathbb{Z}_2 \).

On the other hand, for each \( i \in I \), \( S_i \) is a domain and, as well, \( {\rm M}_n(S_i) \) is a GWNC ring, Lemma \ref{lemma 2.4}, insures that \( J(S_i) = \{0\} \). Thus, for each \( i \in I \), \( S_i \cong \mathbb{Z}_2 \). Hence, we see that $R/J(R)$ is a subring of the Boolean ring $\prod S_i$. So, finally $R/J(R)$ is a Boolean ring, as promised.
\end{proof}

We, thereby, yield:

\begin{corollary}\label{corollary 2.42}
Let $R$ be a $2$-primal ring and $n \ge 3$. Then, ${\rm M}_n(R)$ is GWNC if, and only if, $R$ is a strongly nil-clean ring.
\end{corollary}

As is well-known, a ring $R$ is called {\it NI} if ${\rm Nil}(R)$ is an ideal of $R$.

\medskip

The next series of statements somewhat describes the structure of GWNC rings.

\begin{proposition}\label{proposition 2.43}
Let $R$ be an NI ring and $n \ge 3$. Then, ${\rm M}_n(R)$ is GWNC if, and only if, $R/J(R)$ is Boolean and $J({\rm M}_n(R))$ is nil.
\end{proposition}

\begin{proof}
$(\Longleftarrow)$. Assume \( R/J(R) \) is Boolean and \( J({\rm M}_n(R)) \) is nil. Then, an appeal to \cite[Corollary 6]{24} assures that \( {\rm M}_n(R/J(R)) \cong {\rm M}_n(R)/J({\rm M}_n(R)) \) is a nil-clean ring, so it is a GWNC ring. On the other side, since \( J(M_n(R)) \) is nil, Lemma \ref{lemma 2.4} ensures that \( {\rm M}_n(R) \) is a GWNC ring.\\
$(\Longrightarrow)$. Assume \( {\rm M}_n(R) \) is a GWNC ring. Then, owing to Lemma \ref{lemma 2.4}, we have that \( J({\rm M}_n(R)) \) is nil, which forces that \( J(R) \) is nil. Since \( R \) is an NI ring, we have \( \text{Nil}(R) = J(R) \). Therefore, the ring \( R/J(R) \) is reduced and thus $2$-primal. Hence, \( {\rm M}_n(R/J(R)) \) is a GWNC ring as \( R/J(R) \) is $2$-primal. So, Proposition \ref{proposition 2.41} guarantees that \( R/J(R) \) is Boolean.
\end{proof}

\begin{proposition}\label{proposition 2.44}
Let $R$ be an abelian ring and $n \ge 3$. Then, ${\rm M}_n(R)$ is GWNC if, and only if, $R/J(R)$ is Boolean and $J({\rm M}_n(R))$ is nil.
\end{proposition}

\begin{proof}
If \( R/J(R) \) is Boolean and \( J({\rm M}_n(R)) \) is a nil-ideal, then \cite[Corollary 6.5]{26} implies that \( {\rm M}_n(R) \) is a nil-clean ring and thus is GWNC.

Reciprocally, assume that \( {\rm M}_n(R) \) is a GWNC ring. Then, according to Lemma \ref{lemma 2.4}, \( J({\rm M}_n(R)) = {\rm M}_n(J(R)) \) is a nil-ideal. To illustrate that \( R/J(R) \) is Boolean, we first establish that \( R \) is a weakly clean ring. Since \( {\rm M}_n(R) \) is a GWNC ring, in accordance with Corollary \ref{corollary 2.3}, \( {\rm M}_n(R) \) is a weakly clean ring and thus is weakly exchange. Therefore, \cite[Proposition 2.1]{DCE} is applicable to infer that \( R \) is a weakly exchange ring. However, since \( R \) is an abelian ring, in virtue of \cite[Theorem 2.1]{20}, we conclude that \( R \) is a weakly clean ring, as wanted.

Furthermore, since \( R \) is an abelian weakly clean ring, it follows from \cite[Proposition 14]{kosan1} that, for each left primitive ideal \( I \), we have \( R/I \cong {\rm M}_m(D) \), where \( 1 \le m \le 2 \) and \( D \) is a division ring. We prove that \( D \cong \mathbb{Z}_2 \). In fact, since \( {\rm M}_n(R) \) is a GWNC ring, Corollary \ref{corollary 2.9} is a guarantor that \( {\rm M}_n(R/I) \) is a GWNC ring too. Since \( n \ge 3 \), Corollary \ref{corollary 2.34} helps us to conclude that \( D \cong \mathbb{Z}_2 \). But, \cite[Theorem 12.5]{1} gives that \( R/J(R) \) is a subdirect product of primitive rings, so that \( R/J(R) \) is a subdirect product of the \( {\rm M}_n(\mathbb{Z}_2) \), where \( 1 \le m \le 2 \). In the other vein, since \( R \) is abelian and \( J(R) \) is nil, \cite[Corollary 2.5]{29} means that \( R/J(R) \) is abelian, so \( R/J(R) \) is a subdirect product of \( \mathbb{Z}_2 \), thus \( R/J(R) \) is Boolean, as desired.
\end{proof}

A ring $R$ is called {\it NR}, provided ${\rm Nil}(R)$ is a subring of $R$.

\begin{proposition}\label{NR}
Let $R$ be an NR ring and $n \ge 3$. Then, ${\rm M}_n(R)$ is GWNC if, and only if, $R/J(R)$ is Boolean and $J({\rm M}_n(R))$ is nil.
\end{proposition}

\begin{proof}
$(\Longleftarrow)$. The proof is similar to the proof of Proposition \ref{proposition 2.43}.

$(\Longrightarrow)$. It suffices to show that \( R/J(R) \) is Boolean. To that aim, assume \( R \) is an NR ring. Since \( {\rm M}_n(R) \) is a GWNC ring, we have that \( J(R) \) is nil. Thus, by \cite[Proposition 2.16]{chen1}, it follows that \( R/J(R) \) is abelian. Therefore, via Proposition \ref{proposition 2.44}, we have that \( R/J(R) \) is Boolean.
\end{proof}

Three more consequences are as follows:

\begin{corollary}\label{corollary 2.45}
Let $R$ be a local ring and $n \ge 3$. Then, ${\rm M}_n(R)$ is GWNC if, and only if, $R/J(R) \cong \mathbb{Z}_2$ and $J({\rm M}_n(R))$ is nil.
\end{corollary}

\begin{proof}
$(\Longleftarrow)$. It is clear.\\
$(\Longrightarrow)$. It is enough to demonstrate only that $R/J(R) \cong \mathbb{Z}_2$. Indeed, since ${\rm M}_n(R)$ is a GWNC ring, we discover that ${\rm M}_n(R/J(R)) \cong {\rm M}_n(R)/J({\rm M}_n(R))$ is a GWNC ring as well. And since $R$ is local, $R/J(R)$ is a division ring. Therefore, Corollary \ref{corollary 2.34} insures that $R/J(R) \cong \mathbb{Z}_2$.
\end{proof}

\begin{corollary}\label{corollary 2.46}
Let $R$ be a reduced ring and $n \ge 3$. Then, ${\rm M}_n(R)$ is GWNC if, and only if, $R$ is Boolean.
\end{corollary}

\begin{proof}
$(\Longleftarrow)$. It follows directly from \cite[Corollary 6]{24}.\\
$(\Longrightarrow)$. As $R$ is reduced, $R$ is $2$-primal. However, we receive $J(R)={\rm Nil}(R)=\{0\}$. Then, the result follows from Proposition \ref {proposition 2.41}.	
\end{proof}

An element $r$ of a ring $R$ is called {\it regular} if there exists an element $x\in R$ such that $r = rxr$. Moreover, if every element in a ring is regular, then we call it a {\it regular ring}. A ring in which, for every $r \in R$, there is $x \in R$ such that $r^2x=r$ is called {\it strongly regular}.

\begin{corollary}\label{corollary 2.47}
Let $R$ be a strongly regular ring and $n \ge 3$. Then, ${\rm M}_n(R)$ is GWNC if, and only if, $R$ is Boolean.
\end{corollary}

\begin{proof}
$(\Longleftarrow)$. It is direct from \cite[Corollary 6]{24}.\\
$(\Longrightarrow)$. It is well known that every strongly regular ring is a subdirect product of division rings (see, e.g., \cite{1}). Then, ${\rm M}_n(R)$ is a subdirect product of matrix rings over division rings (cf. \cite{1}). By virtue of Corollary \ref {corollary 2.9}, we deduce that each such matrix ring is GWNC, hence Corollary \ref {corollary 2.34} allows us to infer that every division ring is isomorphic to $\mathbb{Z}_2$. Thus, $R$ must be Boolean, as asserted.
\end{proof}

\begin{lemma}\label{lemma 2.48}
Let $R$ be a ring such that $R = S + K$, where $S$ is a subring of $R$ and $K$ is a nil-ideal of $R$. Then, $S$ is GWNC if, and only if, $R$ is GWNC.
\end{lemma}

\begin{proof}
We know that, $S \cap K \subseteq K$ is a nil-ideal of $S$. Also, we can write that $$R/K = (S+K)/K \cong S/(S \cap K).$$ Therefore, Proposition \ref{proposition 2.8} is applicable inferring the desired result.
\end{proof}

Let $A$, $B$ be two rings, and $M$, $N$ be $(A,B)$-bi-module and $(B,A)$-bi-module, respectively. Also, we consider the bilinear maps $\phi :M\otimes_{B}N\rightarrow A$ and $\psi:N\otimes_{A}M\rightarrow B$ that apply to the following properties.
$${\rm Id}_{M}\otimes_{B}\psi =\phi \otimes_{A} {\rm Id}_{M},{\rm Id}_{N}\otimes_{A}\phi =\psi \otimes_{B} {\rm Id}_{N}.$$
For $m\in M$ and $n\in N$, define $mn:=\phi (m\otimes n)$ and $nm:=\psi (n\otimes m)$. Now, the $4$-tuple $R=\begin{pmatrix}
	A & M\\
	N & B
\end{pmatrix}$ becomes to an associative ring with obvious matrix operations that is called a {\it Morita context ring}. Denote two-side ideals ${\rm Im} \phi$ and ${\rm Im} \psi$ to $MN$ and $NM$, respectively, that are called the {\it trace ideals} of the Morita context (compare also with \cite{15}).

\begin{proposition}\label{proposition 2.49}
Let $R=\left(\begin{array}{ll}A & M \\ N & B\end{array}\right)$ be a Morita context ring such that $MN$ and $NM$ are nilpotent ideals of $A$ and $B$, respectively. If $R$ is a GWNC ring, then $A$ and $B$ are weakly nil-clean rings. The converse holds provided one of the $A$ or $B$ is nil-clean and the other is weakly nil-clean.
\end{proposition}

\begin{proof}
Apparently, since \( MN \subseteq J(A) \) and \( NM \subseteq J(B) \), by using \cite[Lemma 3.1(1)]{30}, we have
\( J(R) = \begin{pmatrix} J(A) & M \\ N & J(B)\end{pmatrix} \)
and \( R/J(R) \cong A/J(A) \times B/J(B) \). Since \( R \) is a GWNC ring, a consultation with Corollary \ref{corollary 2.9} assures that that \( R/J(R) \) is also GWNC. Therefore, the exploitation of Proposition \ref{proposition 2.20} gives that \( A/J(A) \) and \( B/J(B) \) are weakly nil-clean. Moreover, since \( J(R) \) is nil, we infer that both \( J(A) \) and \( J(B) \) are nil too. Hence, from \cite[Lemma 1]{12}, we conclude that \( A \) and \( B \) are weakly nil-clean.

As for the converse, let us assume that $A$ or $B$ is nil-clean and the other is weakly nil-clean. We have \( R = S + K \), where
\( S = \begin{pmatrix} A & 0 \\ 0 & B \end{pmatrix} \)
is a subring of \( R \) and
\( K = \begin{pmatrix} MN & M \\ N & NM \end{pmatrix} \)
is a nil-ideal of \( R \) since
\[ K^{2l} = \begin{pmatrix} (MN)^l & (MN)^lM \\ (NM)^lN & (NM)^l \end{pmatrix} \]
for every \( l \in \mathbb{N} \). Furthermore, as \( S \cong A \times B \), Proposition \ref{proposition 2.21} enables us that \( S \) is a GWNC ring. Therefore, knowing Lemma \ref{lemma 2.48}, we deduce that \( R \) is a GNC ring as well.	
\end{proof}

Now, let $R$, $S$ be two rings, and let $M$ be an $(R,S)$-bi-module such that the operation $(rm)s = r(ms$) is valid for all $r \in R$, $m \in M$ and $s \in S$. Given such a bi-module $M$, we can put

$$
{\rm T}(R, S, M) =
\begin{pmatrix}
	R& M \\
	0& S
\end{pmatrix}
=
\left\{
\begin{pmatrix}
	r& m \\
	0& s
\end{pmatrix}
: r \in R, m \in M, s \in S
\right\},
$$
where this set forms a ring with the usual matrix operations. The so-stated formal matrix ${\rm T}(R, S, M)$ is called a {\it formal triangular matrix ring}. In Proposition \ref{proposition 2.49}, if we set $N =\{0\}$, then we will obtain the following claim.

\begin{corollary}\label{corollary 2.50}
Let $R,S$ be rings and let $M$ be an $(R,S)$-bi-module. If the formal triangular matrix ring ${\rm T}(R,S,M)$ is GWNC, then $R$, $S$ are weakly nil-clean. The converse holds if one of the rings $R$ or $S$ is nil-clean and the other is weakly nil-clean.
\end{corollary}

Given a ring $R$ and a central elements $s$ of $R$, the $4$-tuple $\begin{pmatrix}
	R & R\\
	R & R
\end{pmatrix}$ becomes a ring with addition component-wise and with multiplication defined by
$$\begin{pmatrix}
	a_{1} & x_{1}\\
	y_{1} & b_{1}
\end{pmatrix}\begin{pmatrix}
	a_{2} & x_{2}\\
	y_{2} & b_{2}
\end{pmatrix}=\begin{pmatrix}
	a_{1}a_{2}+sx_{1}y_{2} & a_{1}x_{2}+x_{1}b_{2} \\
	y_{1}a_{2}+b_{1}y_{2} & sy_{1}x_{2}+b_{1}b_{2}
\end{pmatrix}.$$
This ring is denoted by ${\rm K}_{s}(R)$. A Morita context
$\begin{pmatrix}
	A & M\\
	N & B
\end{pmatrix}$ with $A=B=M=N=R$ is called a {\it generalized matrix ring} over $R$. It was observed by Krylov in \cite{16} that a ring $S$ is a generalized matrix ring over $R$ if, and only if, $S={\rm K}_{s}(R)$ for some $s\in {\rm Z}(R)$. Here $MN=NM=sR$, so $MN\subseteq J(A)\Longleftrightarrow s\in J(R)$, $NM\subseteq J(B)\Longleftrightarrow s\in J(R)$, and $MN$, $NM$  are nilpotent $\Longleftrightarrow s$ is a nilpotent.

\begin{corollary}\label{corollary 2.51}
Let $R$ be a ring and $s\in {\rm Z}(R)\cap {\rm Nil}(R)$. If ${\rm K}_{s}(R)$ is a GWNC ring, then $R$ is a weakly nil-clean ring. The converse holds, provided $R$ is a nil-clean ring.
\end{corollary}

Imitating Tang and Zhou (cf. \cite{17}), for $n\geq 2$ and for $s\in {\rm Z}(R)$, the $n\times n$ formal matrix ring over $R$ defined by $s$, and denoted by ${\rm M}_{n}(R;s)$, is the set of all $n\times n$ matrices over $R$ with usual addition of matrices and with multiplication defined below:

\noindent For $(a_{ij})$ and $(b_{ij})$ in ${\rm M}_{n}(R;s)$, set
$$(a_{ij})(b_{ij})=(c_{ij}), \quad \text{where} ~~ (c_{ij})=\sum s^{\delta_{ikj}}a_{ik}b_{kj}.$$
Here, $\delta_{ijk}=1+\delta_{ik}-\delta_{ij}-\delta_{jk}$, where $\delta_{jk}$, $\delta_{ij}$, $\delta_{ik}$ are the Kroncker delta symbols.

\medskip

We, thus, come to the following.

\begin{corollary}\label{corollary 2.52}
Let $R$ be a ring and $s\in {\rm Z}(R)\cap {\rm Nil}(R)$. If ${\rm M}_{n}(R;s)$ is a GWNC ring, then $R$ is a weakly nil-clean ring. The converse holds, provided $R$ is a nil-clean ring.
\end{corollary}

\begin{proof}
If $n = 1$, then ${\rm M}_n(R;s) = R$. So, in this case, there is nothing to prove. Let $n=2$. By the definition of ${\rm M}_n(R;s)$, we have ${\rm M}_2 (R;s) \cong {\rm K}_{s^2} (R)$. Apparently, $s^2 \in {\rm Nil} (R) \cap {\rm Z} (R)$, so the claim holds for $n = 2$ with the help of Corollary \ref{corollary 2.51}.
	
To proceed by induction, assume now that $n>2$ and that the assertion holds for ${\rm M}_{n-1} (R;s)$. Set $A := {\rm M}_{n-1} (R;s)$. Then, ${\rm M}_n (R;s) =
	\begin{pmatrix}
		A & M\\
		N & R
	\end{pmatrix}$
	is a Morita context, where $$M =
	\begin{pmatrix}
		M_{1n}\\
		\vdots\\
		M_{n-1, n}
	\end{pmatrix}
	\quad \text{and} \quad  N = (M_{n1} \dots M_{n, n-1})$$ with $M_{in} = M_{ni} = R$ for all $i = 1, \dots, n-1,$ and
\begin{align*}
		&\psi: N \otimes M \rightarrow N, \quad n \otimes m \mapsto snm\\
		&\phi : M \otimes N \rightarrow M, \quad  m \otimes n \mapsto smn.
\end{align*}
Besides, for $x =
\begin{pmatrix}
		x_{1n}\\
		\vdots\\
		x_{n-1, n}
\end{pmatrix}
\in M$ and $y = (y_{n1} \dots y_{n, n-1}) \in N$, we write $$xy =
\begin{pmatrix}
		s^2x_{1n}y_{n1} & sx_{1n}y_{n2} & \dots & sx_{1n}y_{n, n-1}\\
		sx_{2n}y_{n1} & s^2x_{2n}y_{n2} & \dots & sx_{2n}y_{n, n-1}\\
		\vdots & \vdots &\ddots & \vdots\\
		sx_{n-1, n}y_{n1} & sx_{n-1, n}y_{n2} & \dots & s^2x_{n-1, n}y_{n, n-1}
\end{pmatrix} \in sA$$ and $$yx = s^2y_{n1}x_{1n} + s^2y_{n2}x_{2n} + \dots + s^2y_{n, n-1}x_{n-1, n} \in s^2 R.$$ Since $s$ is nilpotent, we see that $MN$ and $NM$ are nilpotent too. Thus, we obtain that $$\frac{{\rm M}_n (R; s)}{J({\rm M}_n (R; s))} \cong \frac{A}{J (A)} \times \frac{R}{J (R)}.$$ Finally, the induction hypothesis and Proposition \ref{proposition 2.49} yield the claim after all.
\end{proof}

A Morita context $\begin{pmatrix}
	A & M\\
	N & B
\end{pmatrix}$ is called {\it trivial}, if the context products are trivial, i.e., $MN=0$ and $NM=0$. We now have
$$\begin{pmatrix}
	A & M\\
	N & B
\end{pmatrix}\cong {\rm T}(A\times B, M\oplus N),$$
where
$\begin{pmatrix}
	A & M\\
	N & B
\end{pmatrix}$ is the trivial Morita context by consulting with \cite{18}.

An other consequence is the following.

\begin{corollary}\label{corollary 2.53}
If the trivial Morita context
$\begin{pmatrix}
		A & M\\
		N & B
\end{pmatrix}$ is a GWNC ring, then $A$, $B$ are weakly nil-clean rings. The converse holds if one of the rings $A$ or $B$ is nil-clean and the other is weakly nil-clean.
\end{corollary}

\begin{proof}
It is apparent to see that the isomorphisms
$$\begin{pmatrix}
		A & M\\
		N & B
\end{pmatrix} \cong {\rm T}(A\times B,M\oplus N) \cong \begin{pmatrix}
		A\times B & M\oplus N\\
		0 & A \times B
\end{pmatrix}$$ are fulfilled. Then, the rest of the proof follows by combining Corollary \ref{corollary 2.11} and Proposition \ref{proposition 2.19}.
\end{proof}

\section{GWNC Group Rings}

We are concerned here with the examination of groups rings in which all non-units are weakly nil-clean. To this target, following the traditional terminology, we say that a group $G$ is a {\it $p$-group} if every element of $G$ is a power of the prime number $p$. Moreover, a group $G$ is said to be {\it locally finite} if every finitely generated subgroup is finite.
	
\medskip

Suppose now that $G$ is an arbitrary group and $R$ is an arbitrary ring. As usual, $RG$ stands for the group ring of $G$ over $R$. The homomorphism $\varepsilon :RG\rightarrow R$, defined by $\varepsilon (\displaystyle\sum_{g\in G}a_{g}g)=\displaystyle\sum_{g\in G}a_{g}$, is called the {\it augmentation map} of $RG$ and its kernel, denoted by $\Delta (RG)$, is called the {\it augmentation ideal} of $RG$.

\medskip

Before receiving our major assertion of this section, we start our considerations with the next few preliminaries.

\begin{lemma} \label{R is GWNC}
If $RG$ is a GWNC ring, then $R$ is too GWNC.
\end{lemma}

\begin{proof}
We know that \( RG/\Delta(RG) \cong R \). Therefore, in virtue of Corollary 2.9, it follows that \( R \) must be a GWNC ring, as stated.
\end{proof}

\begin{lemma}\label{rg}
Let $R$ be a GWNC ring with $p \in {\rm Nil}(R)$ and let $G$ be a locally finite $p$-group, where $p$ is a prime. Then, the group ring $RG$ is GWNC.
\end{lemma}

\begin{proof}
In accordance with \cite[Proposition 16]{con}, we know that \( \Delta(RG) \) is a nil-ideal. Thus, since \( \Delta(RG) \) is nil and \( RG/\Delta(RG) \cong R \), Proposition 2.8(i) allows us to infer that \( RG \) is a GWNC ring.
\end{proof}

According to Lemma \ref{R is GWNC}, if $RG$ is a GWNC ring, then $R$ is also a GWNC ring. In what follows, we will focus on the topic of what properties the group $G$ will have when $RG$ is a GWNC ring. Before formulating the chief results, we need a series of preliminary technical claims.

\medskip

Explicitly, we obtain the following.

\begin{lemma} \label{2 and 6}
Suppose \( R \) is a GWNC ring. Then either \( 2 \in U(R) \) or \( 2 \in {\rm Nil}(R) \) or \( 6 \in {\rm Nil}(R) \).
\end{lemma}

\begin{proof}
Assume for a moment that \( 2 \notin U(R) \). Then, there exists \( e \in \text{Id}(R) \) and \( q \in \text{Nil}(R) \) such that \( 2 = q \pm e \). Note that \( eq = qe \), because $2$ is a central element. If \( 2 = e + q \), then \( 1 - e = q - 1 \in \text{Id}(R) \cap U(R) \). Thus, \( e = 0 \), which implies \( 2 = q \in \text{Nil}(R) \). Now, if \( 2 = -e + q \), we have \( 4 = e + p \) for some \( p \in \text{Nil}(R) \). Hence, \( 6 = 4 + 2 = p + q \in \text{Nil}(R) \) by noting that \( pq = qp \).
\end{proof}

\begin{lemma} \label{gwnc and gnc and wnc}
Suppose \( R \) is a ring such that \( 2 \notin U(R) \). Then the following conditions are equivalent:
	
(1) \( R \) is a GWNC ring.
	
(2) Either \( R \) is a GNC ring or \( R \) is weakly nil-clean.
\end{lemma}

\begin{proof}
\( (2) \Longrightarrow (1) \). It is obvious, so we drop off the details.
	
\( (1) \Longrightarrow (2) \). Mimicking Lemma \ref{2 and 6}, we have that either \( 2 \in \text{Nil}(R) \) or \( 6 \in \text{Nil}(R) \). If \( 2 \in \text{Nil}(R) \), it is clear that \( R \) is a GNC ring. If \( 6 \in \text{Nil}(R) \) and, for \( n \in \mathbb{N} \), we have \( 6^n = 0 \), then \( R \cong R_1 \oplus R_2 \), where \( R_1 = R / 2^nR \) and \( R_2 = R / 3^nR \). However, Proposition~\ref{proposition 2.19} tells us that \( R_1 \) and \( R_2 \) are weakly nil-clean rings. Moreover, since \( 2 \in \text{Nil}(R_1) \), \( R_1 \) is a nil-clean ring. Thus, \cite[Proposition 3]{12} implies that \( R \) is a weakly nil-clean ring, as required.
\end{proof}

Let us now remember that a ring $R$ is said to be an {\it IU ring} if, for any $a \in R$, either $a$ or $-a$ is the sum of an involution and a unipotent.

\begin{lemma}\cite[Lemma 4.2]{CSH} \label{shaibani}
Let $R$ be a ring. Then, the following are equivalent:
	
(1) $R$ is an IU ring.
	
(2) $R$ is weakly nil-clean and $2 \in U(R)$.
\end{lemma}

We now can attack the truthfulness of the following key statement.

\begin{theorem}\label{gr}
Let $R$ be a ring such that $2 \notin U(R)$, and $G$ a non-trivial abeliean group such that $RG$ is a GWNC ring. Then, $G$ is a $2$-group, where $2$ belongs to ${\rm Nil}(R)$.
\end{theorem}

\begin{proof}
Suppose \( RG \) is a GWNC ring. From Lemma \ref{gwnc and gnc and wnc}, we have that either \( RG \) is a GNC ring or \( RG \) is a weakly nil-clean ring. If, foremost, \( RG \) is a GNC ring, one concludes from \cite[Theorem 3.8]{8} that \( G \) is a \( p \)-group, where \( p \in \text{Nil}(R) \). If \( p \) is odd, this obviously contradicts \( 2 \in U(R) \), so it must be that \( p = 2 \).

If \( RG \) is a weakly nil-clean ring, then \cite[Theorem 1.14]{roh} riches us that either \( R \) is a nil-clean ring and \( G \) is a $2$-group, or \( R \) is a UI ring and \( G \) is a $3$-group. If \( R \) is a nil-clean ring and \( G \) is a 2-group, there is nothing left to prove, because from \cite[Proposition 3.14]{4}, we have \( 2 \in \text{Nil}(R) \). However, if \( R \) is a UI ring and \( G \) is a 3-group, from Lemma \ref{shaibani}, we have \( 2 \in U(R) \), which is a contradiction.
\end{proof}

\section{Open Questions}

We finish our work with the following two questions which allude us.

\begin{problem}\label{1}
Examine those rings whose non-invertible elements are strongly weakly nil-clean.
\end{problem}

A ring $R$ is called {\it uniquely weakly nil-clean}, provided that $R$ is a weakly nil-clean ring in which every nil-clean element is uniquely nil-clean.

\begin{problem}\label{2}
Examine those rings whose non-invertible elements are uniquely weakly nil-clean.
\end{problem}

\medskip
\medskip

\noindent {\bf Funding:} The work of the first-named author, P.V. Danchev, is partially supported by the Junta de Andaluc\'ia, Grant FQM 264. All other three authors are supported by the National Elite Foundation of Islamic Republic of Iran and receive funds from this foundation.

\end{document}